\documentclass[11pt]{article}
\usepackage{amsmath}
\usepackage{amssymb}
\newcommand*{\affaddr}[1]{#1} 
\newcommand*{\affmark}[1][*]{\textsuperscript{#1}}

\font\sevenmsx=msam7 \font\fivemsx=msam5 \font\tenmsy=msbm10
\font\sevenmsy=msbm7 \font\fivemsy=msbm5
\RequirePackage[dvips]{graphicx}
\newfam\msxfam
\newfam\msyfam
\textfont\msxfam=\tenmsx  \scriptfont\msxfam=\sevenmsx
  \scriptscriptfont\msxfam=\fivemsx
\textfont\msyfam=\tenmsy  \scriptfont\msyfam=\sevenmsy
  \scriptscriptfont\msyfam=\fivemsy
\def\hexnumber@#1{\ifnum#1<10 \number#1\else
\ifnum#1=10 A\else\ifnum#1=11 B\else\ifnum#1=12 C\else \ifnum#1=13
D\else\ifnum#1=14 E\else\ifnum#1=15 F\fi\fi\fi\fi\fi\fi\fi}

\def\msx@{\hexnumber@\msxfam}
\def\msy@{\hexnumber@\msyfam}
\mathchardef\boxdot="2\msx@00 \mathchardef\boxplus="2\msx@01
\mathchardef\boxtimes="2\msx@02 \mathchardef\square="0\msx@03
\mathchardef\blacksquare="0\msx@04
\mathchardef\centerdot="2\msx@05 \mathchardef\lozenge="0\msx@06
\mathchardef\blacklozenge="0\msx@07
\mathchardef\circlearrowright="3\msx@08
\mathchardef\circlearrowleft="3\msx@09
\mathchardef\rightleftharpoons="3\msx@0A
\mathchardef\leftrightharpoons="3\msx@0B
\mathchardef\boxminus="2\msx@0C \mathchardef\Vdash="3\msx@0D
\mathchardef\Vvdash="3\msx@0E \mathchardef\vDash="3\msx@0F
\mathchardef\twoheadrightarrow="3\msx@10
\mathchardef\twoheadleftarrow="3\msx@11
\mathchardef\leftleftarrows="3\msx@12
\mathchardef\rightrightarrows="3\msx@13
\mathchardef\upuparrows="3\msx@14
\mathchardef\downdownarrows="3\msx@15
\mathchardef\upharpoonright="3\msx@16

\mathchardef\downharpoonright="3\msx@17
\mathchardef\upharpoonleft="3\msx@18
\mathchardef\downharpoonleft="3\msx@19
\mathchardef\rightarrowtail="3\msx@1A
\mathchardef\leftarrowtail="3\msx@1B
\mathchardef\leftrightarrows="3\msx@1C
\mathchardef\rightleftarrows="3\msx@1D \mathchardef\Lsh="3\msx@1E
\mathchardef\Rsh="3\msx@1F \mathchardef\rightsquigarrow="3\msx@20
\mathchardef\leftrightsquigarrow="3\msx@21
\mathchardef\looparrowleft="3\msx@22
\mathchardef\looparrowright="3\msx@23
\mathchardef\circeq="3\msx@24 \mathchardef\succsim="3\msx@25
\mathchardef\gtrsim="3\msx@26 \mathchardef\gtrapprox="3\msx@27
\mathchardef\multimap="3\msx@28 \mathchardef\therefore="3\msx@29
\mathchardef\because="3\msx@2A \mathchardef\doteqdot="3\msx@2B

\mathchardef\triangleq="3\msx@2C \mathchardef\precsim="3\msx@2D
\mathchardef\lesssim="3\msx@2E \mathchardef\lessapprox="3\msx@2F
\mathchardef\eqslantless="3\msx@30
\mathchardef\eqslantgtr="3\msx@31
\mathchardef\curlyeqprec="3\msx@32
\mathchardef\curlyeqsucc="3\msx@33
\mathchardef\preccurlyeq="3\msx@34 \mathchardef\leqq="3\msx@35
\mathchardef\leqslant="3\msx@36 \mathchardef\lessgtr="3\msx@37
\mathchardef\backprime="0\msx@38
\mathchardef\risingdotseq="3\msx@3A
\mathchardef\fallingdotseq="3\msx@3B
\mathchardef\succcurlyeq="3\msx@3C \mathchardef\geqq="3\msx@3D
\mathchardef\geqslant="3\msx@3E \mathchardef\gtrless="3\msx@3F
\mathchardef\sqsubset="3\msx@40 \mathchardef\sqsupset="3\msx@41
\mathchardef\vartriangleright="3\msx@42
\mathchardef\vartriangleleft="3\msx@43
\mathchardef\trianglerighteq="3\msx@44
\mathchardef\trianglelefteq="3\msx@45
\mathchardef\bigstar="0\msx@46 \mathchardef\between="3\msx@47
\mathchardef\blacktriangledown="0\msx@48
\mathchardef\blacktriangleright="3\msx@49
\mathchardef\blacktriangleleft="3\msx@4A
\mathchardef\vartriangle="3\msx@4D
\mathchardef\blacktriangle="0\msx@4E
\mathchardef\triangledown="0\msx@4F \mathchardef\eqcirc="3\msx@50
\mathchardef\lesseqgtr="3\msx@51 \mathchardef\gtreqless="3\msx@52
\mathchardef\lesseqqgtr="3\msx@53
\mathchardef\gtreqqless="3\msx@54
\mathchardef\Rrightarrow="3\msx@56
\mathchardef\Lleftarrow="3\msx@57 \mathchardef\veebar="2\msx@59
\mathchardef\barwedge="2\msx@5A
\mathchardef\doublebarwedge="2\msx@5B \mathchardef\angle="0\msx@5C
\mathchardef\measuredangle="0\msx@5D
\mathchardef\sphericalangle="0\msx@5E
\mathchardef\varpropto="3\msx@5F \mathchardef\smallsmile="3\msx@60
\mathchardef\smallfrown="3\msx@61 \mathchardef\Subset="3\msx@62
\mathchardef\Supset="3\msx@63 \mathchardef\Cup="2\msx@64

\mathchardef\Cap="2\msx@65

\mathchardef\curlywedge="2\msx@66 \mathchardef\curlyvee="2\msx@67
\mathchardef\leftthreetimes="2\msx@68
\mathchardef\rightthreetimes="2\msx@69
\mathchardef\subseteqq="3\msx@6A \mathchardef\supseteqq="3\msx@6B
\mathchardef\bumpeq="3\msx@6C \mathchardef\Bumpeq="3\msx@6D
\mathchardef\lll="3\msx@6E

\mathchardef\ggg="3\msx@6F

\mathchardef\circledS="0\msx@73 \mathchardef\pitchfork="3\msx@74
\mathchardef\dotplus="2\msx@75 \mathchardef\backsim="3\msx@76
\mathchardef\backsimeq="3\msx@77 \mathchardef\complement="0\msx@7B
\mathchardef\intercal="2\msx@7C \mathchardef\circledcirc="2\msx@7D
\mathchardef\circledast="2\msx@7E
\mathchardef\circleddash="2\msx@7F
\def\ulcorner{\delimiter"4\msx@70\msx@70 }
\def\urcorner{\delimiter"5\msx@71\msx@71 }
\def\llcorner{\delimiter"4\msx@78\msx@78 }
\def\lrcorner{\delimiter"5\msx@79\msx@79 }
\def\yen{\mathhexbox\msx@55 }
\def\checkmark{\mathhexbox\msx@58 }
\def\circledR{\mathhexbox\msx@72 }
\def\maltese{\mathhexbox\msx@7A }
\mathchardef\lvertneqq="3\msy@00 \mathchardef\gvertneqq="3\msy@01
\mathchardef\nleq="3\msy@02 \mathchardef\ngeq="3\msy@03
\mathchardef\nless="3\msy@04 \mathchardef\ngtr="3\msy@05
\mathchardef\nprec="3\msy@06 \mathchardef\nsucc="3\msy@07
\mathchardef\lneqq="3\msy@08 \mathchardef\gneqq="3\msy@09
\mathchardef\nleqslant="3\msy@0A \mathchardef\ngeqslant="3\msy@0B
 \mathchardef\lneq="3\msy@0C \mathchardef\gneq="3\msy@0D
\mathchardef\npreceq="3\msy@0E \mathchardef\nsucceq="3\msy@0F
\mathchardef\precnsim="3\msy@10 \mathchardef\succnsim="3\msy@11
\mathchardef\lnsim="3\msy@12 \mathchardef\gnsim="3\msy@13
\mathchardef\nleqq="3\msy@14 \mathchardef\ngeqq="3\msy@15
\mathchardef\precneqq="3\msy@16 \mathchardef\succneqq="3\msy@17
\mathchardef\precnapprox="3\msy@18
\mathchardef\succnapprox="3\msy@19 \mathchardef\lnapprox="3\msy@1A
\mathchardef\gnapprox="3\msy@1B \mathchardef\nsim="3\msy@1C
\mathchardef\napprox="3\msy@1D \mathchardef\varsubsetneq="3\msy@20
\mathchardef\varsupsetneq="3\msy@21
\mathchardef\nsubseteqq="3\msy@22
\mathchardef\nsupseteqq="3\msy@23
\mathchardef\subsetneqq="3\msy@24
\mathchardef\supsetneqq="3\msy@25
\mathchardef\varsubsetneqq="3\msy@26
\mathchardef\varsupsetneqq="3\msy@27
\mathchardef\subsetneq="3\msy@28 \mathchardef\supsetneq="3\msy@29
\mathchardef\nsubseteq="3\msy@2A \mathchardef\nsupseteq="3\msy@2B
\mathchardef\nparallel="3\msy@2C \mathchardef\nmid="3\msy@2D
\mathchardef\nshortmid="3\msy@2E
\mathchardef\nshortparallel="3\msy@2F
\mathchardef\nvdash="3\msy@30 \mathchardef\nVdash="3\msy@31
\mathchardef\nvDash="3\msy@32 \mathchardef\nVDash="3\msy@33
\mathchardef\ntrianglerighteq="3\msy@34
\mathchardef\ntrianglelefteq="3\msy@35
\mathchardef\ntriangleleft="3\msy@36
\mathchardef\ntriangleright="3\msy@37
\mathchardef\nleftarrow="3\msy@38
\mathchardef\nrightarrow="3\msy@39
\mathchardef\nLeftarrow="3\msy@3A
\mathchardef\nRightarrow="3\msy@3B
\mathchardef\nLeftrightarrow="3\msy@3C
\mathchardef\nleftrightarrow="3\msy@3D
\mathchardef\divideontimes="2\msy@3E
\mathchardef\varnothing="0\msy@3F \mathchardef\nexists="0\msy@40
\mathchardef\mho="0\msy@66 \mathchardef\thorn="0\msy@67
\mathchardef\beth="0\msy@69 \mathchardef\gimel="0\msy@6A
\mathchardef\daleth="0\msy@6B \mathchardef\lessdot="3\msy@6C
\mathchardef\gtrdot="3\msy@6D \mathchardef\ltimes="2\msy@6E
\mathchardef\rtimes="2\msy@6F \mathchardef\shortmid="3\msy@70
\mathchardef\shortparallel="3\msy@71
\mathchardef\smallsetminus="2\msy@72
\mathchardef\thicksim="3\msy@73 \mathchardef\thickapprox="3\msy@74
\mathchardef\approxeq="3\msy@75 \mathchardef\succapprox="3\msy@76
\mathchardef\precapprox="3\msy@77
\mathchardef\curvearrowleft="3\msy@78
\mathchardef\curvearrowright="3\msy@79
\mathchardef\digamma="0\msy@7A \mathchardef\varkappa="0\msy@7B
\mathchardef\hslash="0\msy@7D \mathchardef\hbar="0\msy@7E
\mathchardef\backepsilon="3\msy@7F
\def\Bbb{\ifmmode\let\next\Bbb@\else
\def\next{\errmessage{Use \string\Bbb\space only in math mode}}\fi\next}
\def\Bbb@#1{{\Bbb@@{#1}}}
\def\Bbb@@#1{\fam\msyfam#1}

\newtheorem{definition}{\bf Definition}[section]

\newtheorem{theorem}[definition]{\bf Theorem}

\def \be {\begin{equation}}
\def \ee {\end{equation}}
\def \ba {\begin{eqnarray}}
\def \ea {\end{eqnarray}}
\def \bas {\begin{eqnarray*}}
\def \eas {\end{eqnarray*}}

\date{}
\begin{document}


\title{\bf Real Zeros of Algebraic Polynomials with Nonidentical Dependent Random Coefficients}
\author{
Sabita Sahoo\affmark[a] and
Partiswari Maharana\affmark[b,*]\\
\affaddr{\affmark[a] Faculty of Department of Mathematics,}\\
\affaddr{\affmark[b]Department of Mathematics,} 
\thanks{{Corresponding author. E-mail addresses: sabitamath@suniv.ac.in, partiswarimath1@suniv.ac.in}~ }\\
{ Sambalpur University, Odisha - 768019,
India.}
}

\date{ }
\maketitle
\begin{center}
{\bf Abstract}
\end{center}
\hspace*{.2cm} 
The expected number of real zeros of a random algebraic polynomial 
$a_0+a_1x+a_2x^2+a_3x^3+....+a_{n-1}x^{n-1}$ 
depends on the types of random coefficients, with large $n.$ In all works, the coefficients are either independent or dependent but varience of coefficients $a_i$ is one. In these cases the exepected number of real zeros is found out to be asymptotic to  $\frac{2}{\pi}logn.$ In this article, we have considered the 
negatively correlated dependent random coefficients 
$\{a_i\}_{i=0}^{n-1}$ with varience $\sigma^{2i},$ for $\sigma >1$ and coefficient of correlation  $\rho_{ij}=-\rho^{|i-j|},$ where $0<\rho<\frac{1}{3}.$ The expected number of real zeros is found to be  $\frac{2}{\pi \sigma}logn,$ which is depended on 
$\sigma.$
 
Subject Classification-: Primary 60H99; Secondary 65H99.


 \section{Introduction} 
\setcounter{equation}{0}
Study on expected number of real zeros of algebric polynomial with random coefficient started with Kac in 1940's. He considered random polynomials


\begin{equation}\label{1.1}
P_n(x)=\sum_{i=0}^{n-1}a_ix^i,
\end{equation}
where $\{a_i\}_{i=0}^{n-1}$ a sequence of random coefficients. 
Denote $N_n(a,b)$ to be the number of real zeros in the interval $(a,b).$ In the beginning work of Kac \cite{A1}, the coefficients $\{a_i\}_{i=0}^{n-1}$ were considered to be independent 
normally distributed with large $n.$ He found out that the expected number of real zeros $N_n(-\infty, \infty)$ is $\frac{2}{\pi}logn.$
Ibragimov and Maslova \cite{A2} showed that this asymptotic value remains unchanged for most of the distributions. In particular when the coefficients belong to domain of attraction of normal law, then also it has the same asymptotic value. But if the mean of the coefficients is nonzero constant then this asymptotic value reduces to half. It is observed in some cases in \cite{A3} that the expected number of real zeros $EN_n(-\infty, \infty)$ increases to $\sqrt{n},$ when the coefficients are nonidentical i.e. $var(a_i)=\binom {n}{j}.$ On the other hand this expected number of real zero increases to $\frac{n}{\sqrt{3}}$ for the random trigonometric polynomial $a_0+a_1cosx+a_2cos2x+a_3cos3x+....+a_{n-1}cos{(n-1)}x.$ This work was developed by Dunnage \cite{A4}.\\

In comparision to independent case, there is few works related to dependent case. One of the reason is that the analysis for dependent case is complicated. The work for this case was developed by Sambandham \cite{A5,A6,A7} for algebraic polynomials and  Renganathan, Sambandham \cite{A8} for trigonometric polynomials. Sambandham in his beginning work,  
considered two cases of random coefficients which are positively  dependent. In first case he showed that if the correlation $\rho_{ij}=\rho,$ (a constant) where $0<\rho<1,$ then the expected number of real zero is asymptotically equal to half of 
$EN_n(-\infty, \infty).$ For second case he considered the correlation coefficient $\rho_{ij}=\rho^{|i-j|},$ for $0<\rho<\frac{1}{2},$ and showed that this asymptotic value remain unchanged. Their is less study on algebric polynomials with negatively correlated coefficients. Newman \cite{A15} developed investigation when the coefficients are negatively correlated. This work is motivated work of Farahmand and Nezakati \cite{A16}. They have also considered algebric polynomials with random coefficients with mean zero, variance one, negetively correlated and obtained the expected number of real zeros is $\frac{2}{\pi}logn$. 
In this article we obtained real zeros of algebric polynomials with random coefficients $\{a_i\}_{i=0}^{n-1}$  in (\ref{1.1}) to be sequence of dependent random variables with mean zero, variance $(a_i)=\sigma^{2i},$ where $\sigma>1$ is any constant independent of $n,$ and correlation between any two coefficient $\rho_{ij}$ is $-\rho^{|i-j|},$ for $0<\rho<1/3.$


\section{Preliminary Analysis}
\setcounter{equation}{0}
To find expected number of real zeros we use the concept based on Kac\cite{A1} and Rice\cite{A17} results. The Kac-Rice formula to calculate the expected number of real zeros as 
\begin{equation} \label{2.2}
EN_n(a,b)=\frac{1}{\pi} \int_a^b \frac{\Delta}{A^2}dx, 
\end{equation}
where
$$
A^2=var(P_n(x)), B^2=var(P_n^{'}(x))$$
\begin{equation} \label{2.1}
C=cov(P_n(x) P_n^{'}(x)), \quad \Delta^2=A^2B^2-C^2.
\end{equation}

In order to use the Kac-Rice formula to find $EN_n(-\infty,\infty),$
it is observed that changing $x$ to  $\frac{1}{x}$ and $x$ to $-x$ leaves the distribution of the coefficients of $P_n(x)$ in  (\ref{1.1}) invariant. Hence, the expected number of real zeros in the interval 
$(-1,1)$ is asymptotically equal as in $(-\infty,-1)$ and $(1,  \infty).$ Therefore, it suffices to give the result for $EN_n(-1,1)$ only.
\section{Main result}
\setcounter{equation}{0}
In the following theorem the expected number of real zeros of algebric polynomial (\ref{1.1}) where random coefficients $a_i$ are nonidentical and negatively correlated is estimated.
\begin{theorem}
If the random variable $a_i,\{i=0,1,...., n-1 \}$ of the polynomial $P_n(x)$ are assumed to be dependent with mean zero, $var(a_i)=\sigma^{2i}$ and correlation coefficient $\rho_{ij}=-\rho^{|i-j|},$ then for all sufficiently large $n,$ the expected number of real zeros of $P_n(x)$ is 
$$
EN_n(-\infty, \infty) \sim \frac{2}{(\pi \sigma)}logn.
 $$
\end{theorem} 
\begin{proof}
From the assumptions of theorem for the distributions of the coefficients of $P_n(x),$ from (\ref{2.1}) we can obtain   
\begin{equation} \label{3.1}
A^2=\sum_{i=0}^{n-1} x^{2i}\sigma^{2i}-\sum_{i=0}^{n-1}\sum_{\substack{j=0\\ j\neq i}}^{n-1}\rho^{|i-j|}x^{i+j}\sigma^{i+j},
\end{equation}
\\
\begin{equation} \label{3.2}
B^2=\sum_{i=0}^{n-1} i^2 x^{2i-2} \sigma^{2i}-\sum_{i=0}^{n-1}\sum_{\substack{j=0 \\ j\neq i }}^{n-1} ij \rho^{|i-j|}x^{i+j-2}\sigma^{i+j},
\end{equation}
\\
\begin{equation} \label{3.3}
C=\sum_{i=0}^{n-1} i x^{2i-1} \sigma^{2i}-\sum_{i=0}^{n-1}\sum_{\substack{j=0\\ j \neq i}}^{n-1} i \rho^{|i-j|}x^{i+j}\sigma^{i+j}.
\end{equation}
As the values of $A^2, B^2$ and $C$ are obtained in \cite{A16} by the method of Sambandham \cite{A5}, following  we obtain the values of the above identities as\\
$$A^2= \frac{1-x^{2n}\sigma^{2n}}{1-x^{2}\sigma^{2}}-
\frac{\rho({\rho x\sigma}-{(\rho x\sigma)^n})}{(\rho-x\sigma)(1-\rho x\sigma)}
+\frac{\rho({( x\sigma)}^2-{( x\sigma)^{2n}})}{(\rho-x\sigma)(1-x^{2}\sigma^{2})}$$
\begin{equation} \label{3.4}
-\frac{\rho x\sigma({1-{( x\sigma)^{2n-2}})}}{(1-\rho x\sigma)(1-x^{2}\sigma^{2})}
+\frac{\rho^2 x^n \sigma^n ({\rho^{n-1}}-{( x\sigma)^{n-1}})}{(\rho-x\sigma)(1-\rho x\sigma)},
\end{equation}
\\
 $$B^2= \frac{\sigma^2(1+x^{2}\sigma^{2})(1-x^{2n}\sigma^{2n})
-n^2x^{2n-2}\sigma^{2n}(1-x^2\sigma^2)^2-2nx^{2n}\sigma^{2n+2}(1-x^{2}\sigma^{2})}{(1-x^{2}\sigma^{2})^3}$$
$$ - \frac{\rho[\rho \sigma^2-n \rho^n x^{n-1}\sigma^{n+1}+(n-1)\rho^{n+1}x^n \sigma^{n+2}]}{(\rho-x\sigma)^2(1-\rho x\sigma)^2}
+\frac{\rho[x \sigma^3-nx^{2n-1}\sigma^{2n+1}+(n-1)x^{2n+1}\sigma^{2n+3}]}{(\rho-x\sigma)^2(1-x^{2}\sigma^{2})^2}$$ 
$$+\frac{\rho \sigma^2[(1-x^{2}\sigma^{2})(1-n^2x^{2n-2}\sigma^{2n}...)+2x \sigma(x \sigma-x^{2n+1}\sigma^{2n+3}+...)]}{(\rho-x\sigma)(1-x^{2}\sigma^{2})^2}$$
$$-\frac{\rho[x \sigma^3- x^{2n-1}\sigma^{2n+1}+(n-1)x^{2n-1}\sigma^{2n+1}-(n-1)x^{2n-3}\sigma^{2n+1}]}{(1-\rho x\sigma)^2(1-x^{2}\sigma^{2})^2}$$
$$-\frac{\rho x \sigma^3[(1-x^{2}\sigma^{2})(1-nx^{2n-2}\sigma^{2n}...)+2x \sigma(x \sigma- x^{2n-1}\sigma^{2n+1}...)]}{(1-\rho x\sigma)(1-x^{2}\sigma^{2})^3}$$ 
\begin{equation}\label{3.5}
+\frac{[(1-\rho x \sigma)(n\rho^2x^{n-1}\sigma^{n+1})+\rho^{3}x^n\sigma^{n+2}]
[\rho^{n-1}\sigma^2+(n-2)x^{n-1}\sigma^{n+1}-(n-1)\rho x^{n-2}]}
{(1-\rho x\sigma)^2(\rho-x\sigma)^2},
\end{equation}
\\
$$C=\frac{x\sigma^2-x^{2n+1}\sigma^{2n+2}-nx^{2n-1}\sigma^{2n}+nx^{2n+1}\sigma^{2n+2}}{(1-x^{2}\sigma^{2})^2}$$
$$-\frac{\rho(\rho \sigma-n \rho^nx^{n-1} \sigma^n+(n-1)\rho^{n+1}x^n \sigma^{n+1})}{(\rho-x\sigma)(1-\rho x\sigma)^2}$$
$$+\frac{\rho[x \sigma^2-x^{2n+1} \sigma^{2n+2}-nx^{2n-1}\sigma^{2n}+nx^{2n+1}\sigma^{2n+2}]}{(\rho-x\sigma)(1-x^{2}\sigma^{2})^2}$$
$$-\frac{\rho x \sigma[x\sigma^2-x^{2n-1}\sigma^{2n} +(n-1)x^{2n-1}\sigma^{2n}-(n-1)x^{2n-3}\sigma^{2n-2}]}{(1-\rho x\sigma)(1-x^{2}\sigma^{2})^2}$$
\begin{equation} \label{3.6}
+ \frac{\rho^2 x^n \sigma^{n+1}[\rho^{n-1}-x^{n-1}\sigma^{n-1}+(n-1)x^{n-1}\sigma^{n-1}-(n-1)\rho x^{n-2}\sigma^{n-2}]}{(1-\rho x\sigma)(\rho-x\sigma)^2}. 
\end{equation}

Consider
\begin{equation} \label{3.7}
\epsilon=n^{-a}, \quad where \quad
a=1-\frac{loglogn^{10}}{logn} \quad and \quad 0<a<1. 
\end{equation}

Let us first assume $x \in (0,1)$ and divide this interval into two subintervals $(0, \frac{1-\epsilon}{\sigma})$ and $(\frac{1-\epsilon}{\sigma},1).$ 

For $0\leq x\leq \frac{1-\epsilon}{\sigma},$ and all sufficiently large $n,$ it can be show that $$x^n\sigma^n \leq n^{-10}.$$
From the above values in (\ref{3.4})-(\ref{3.6}), we can derive 
\begin{equation}\label{3.8}
A^2=\frac{(1-3\rho x \sigma)[1+O(n^{20})]}{(1-x^2\sigma^2)(1-\rho x \sigma)},
\end{equation} 
 \begin{equation}\label{3.9}
 B^2=\frac{[(1+x^2 \sigma^2)(1-\rho x \sigma)\{2\rho \sigma^2(1+x^2 \sigma^2)-x \sigma^3(1+3 \rho^2)\}
-\rho \sigma^2(1-x^2\sigma^2)^2(1+\rho x \sigma)][1+O(n^{-18})]}{(1-x^2\sigma^2)(1-\rho x \sigma)^2(\rho-x\sigma)^2},
 \end{equation}
\begin{equation} \label{3.10}
C=\frac{(x \sigma^2+3\rho^2 x^3 \sigma^4-3\rho x^2 \sigma^3-\rho \sigma)[1+O(n^{-9})]}{(1-\rho x \sigma)^2(1-x^2\sigma^2)^2}.
\end{equation}
From (\ref{2.1}) and then using (\ref{3.4})-(\ref{3.6}) we can see that 
\begin{equation}\label{3.11}
\frac{\Delta}{A^2}=\frac{\sqrt{A^2B^2-C^2}}{A^2} 
\sim \sqrt{\frac{K(\rho,x,\sigma)}{(\rho-x\sigma)(1-x^2\sigma^2)^2}}
=\sqrt{\frac{|K(\rho,x,\sigma)|}{|(\rho-x\sigma)|(1-x^2\sigma^2)^2}},
\end{equation} 
where 
$$K(\rho, x,\sigma)=\frac{(1+x^2 \sigma^2)(1-\rho x \sigma)\{2\rho \sigma^2(1+x^2 \sigma^2)-x \sigma^2(1+3 \rho^2)\}-\rho \sigma^2(1-x^2 \sigma^2)^2(1+\rho x \sigma)}{(1-\rho x \sigma)(1-3\rho x \sigma)}$$
\begin{equation}\label{3.12}
-\frac{(\rho-x \sigma)(x \sigma^2-3 \rho x^2 \sigma^3+3 \rho^2x^3 \sigma^4-\rho \sigma)}{(1-\rho x \sigma)^2(1-3\rho x \sigma)^2}.
\end{equation}

For evaluating $\frac{\Delta}{A^2}$ and $K(\rho, x, \sigma),$ different method are used in the neighbourhood of $x=\frac{\rho}{\sigma}$ and out side this neighbourhood. We obtain $EN_n$
in the intervals $[0, \frac{\rho}{\sigma}-\frac{1}{n\sigma}], [\frac{\rho}{\sigma}-\frac{1}{n\sigma}, \frac{\rho}{\sigma}+\frac{1}{n\sigma}], [\frac{\rho}{\sigma}+\frac{1}{n\sigma},\frac{1}{2\sigma}],[\frac{1}{2\sigma},\frac{1-\eta}{\sigma}], [\frac{1-\eta}{\sigma},\frac{1-\epsilon}{\sigma}],[\frac{1-\epsilon}{\sigma},1]$ 
where 
\begin{equation}\label{3.13}
\eta=exp\{-(log n)^{\frac{1}{3}}\}.
\end{equation} 

When $0\leq x \leq\frac{\rho}{\sigma}-\frac{1}{n\sigma} ,$ from (\ref{3.11}) and (\ref{3.12}), $K(\rho, x ,\sigma)$ is bounded and $(1-x^2\sigma^2)$ is bounded away from zero. Therefore,
$$\frac{\Delta}{A^2}< \frac{M_0}{(\rho-x \sigma)^{\frac{1}{2}}},$$ 
where $M_0$ is an absolute constant. All $M's$ in the following are also absolute constants.
From (\ref{2.2}), and for all sufficiently large $n,$ we have 
\begin{equation}\label{3.14}
EN_n\left(0,\frac{\rho}{\sigma}-\frac{1}{n\sigma}\right)=O(1).
\end{equation}
Now, assume that $\frac{\rho}{\sigma}-\frac{1}{n\sigma} \leq x \leq \frac{\rho}{\sigma}+\frac{1}{n\sigma}.$ From (\ref{2.1}), (\ref{3.2}) and (\ref{3.8}) we have
 $$
\frac{\Delta}{A^2} <  \left\lbrace\frac{B^2}{A^2}
\right\rbrace^2 < \left\lbrace\frac{ \sum_{i=0}^{n-1}i^2 x^{2i-2}\sigma^{2i}}{ \{(1-3\rho  x \sigma)/[(1-\rho x \sigma)(1-x^2 \sigma^2)]\}\{1+O(n^{-20})\}}\right\rbrace^{\frac{1}{2}}
 $$ 
$$< n \left(\frac{1-\rho x \sigma}{1-3\rho x \sigma}\right)^{\frac{1}{2}}.$$
However, $\left(\frac{1-\rho x \sigma}{1-3\rho x \sigma}\right)^{\frac{1}{2}}$
is bounded. Therefore from (\ref{2.2}) we obtain
\begin{equation}\label{3.15}
EN_n\left(\frac{\rho}{\sigma}-\frac{1}{n\sigma},\frac{\rho}{\sigma}+\frac{1}{n\sigma}\right)=O(1).
\end{equation}

When $\frac{\rho}{\sigma}+\frac{1}{n\sigma} \leq x \leq \frac{1}{2\sigma},$ we find that $K(\rho, x, \sigma)$ and $(1-x^2 \sigma^2)$ 
are both bounded away from zero and 
$$\frac{\Delta}{A^2}< \frac{M_1}{(x \sigma- \rho)^{\frac{1}{2}}}.$$
Therefore, from (\ref{2.2}),
\begin{equation}\label{3.17}
EN_n\left(\frac{\rho}{\sigma}+\frac{1}{n\sigma},\frac{1}{2\sigma}\right)=\frac{M_1}{\pi}\int_{\frac{\rho}{\sigma}+\frac{1}{n\sigma}}^{\frac{1}{2\sigma}}\frac{dx}{(x \sigma- \rho)^{\frac{1}{2}}}=O(1).
\end{equation}

When $\frac{1}{2\sigma}\leq x \leq \frac{1-\eta}{\sigma},$ we obtain 
$\frac{K(\rho, x, \sigma)}{(\rho- x \sigma)}$ is bounded. Then
$$\frac{\Delta}{A^2}< \frac{M_2}{(1-x^2 \sigma^2)}, $$
and hence from (\ref{2.2}) 
\begin{equation}\label{3.18}
EN_n {\left(\frac{1}{2\sigma},{\frac{1-\eta}\sigma} \right)}=\frac{M_2}{\pi}\int_{\frac{1}{2\sigma}}^{\frac{1-\eta}\sigma}\frac{dx}{(1- x^2 \sigma^2)}=O(logn)^{\frac{1}{3}}.
\end{equation}

Now, let $\frac{1-\eta}{\sigma} \leq x \leq \frac{1-\epsilon}{\sigma}.$ From (\ref{3.11}) we get 
\begin{equation}\label{3.19}
\frac{\Delta}{A^2}< \frac{K_1(\rho, x, \sigma)}{(1-x^2 \sigma^2)},
\end{equation}
where $K_1(\rho, x, \sigma)=\sqrt  {\frac{K(\rho, x, \sigma)}{(\rho- x\sigma)}}.$
 In this interval, it can be show that 
\begin{equation}\label{3.20}
K_1^*(\rho,\eta, \epsilon, \sigma)\leq K_1(\rho,\eta, \epsilon, \sigma) \leq K_1^{**}(\rho,\eta, \epsilon, \sigma),
\end{equation}
where $K_1^{**}(\rho,\eta, \epsilon, \sigma)$ and $K_1^*(\rho,\eta, \epsilon, \sigma)$ are $K_1(\rho,x, \sigma)$
 when $x$ is substituated by $\frac{1-\epsilon}{\sigma}$ and $\frac{1-\eta}{\sigma}$ respectively. Therefore from (\ref{2.2}), (\ref{3.7}), (\ref{3.13}) and  (\ref{3.19}) we get,\\
$$\frac{aK_1^{**}(\rho,\eta, \epsilon, \sigma)}{2\pi \sigma}logn+ O(logn)^{\frac{1}{3}} \leq EN_n\left({\frac{1-\eta}\sigma},{\frac{1-\epsilon}\sigma}\right) \leq \frac{aK_1^*(\rho,\eta, \epsilon, \sigma)}{2\pi \sigma}logn+ O(logn)^{\frac{1}{3}},$$\\
since 
$a \rightarrow 1,\quad {K_1^{**}(\rho,\eta, \epsilon, \sigma)}\rightarrow 1, \quad K_1^*(\rho,\eta, \epsilon, \sigma)\rightarrow 1$ as $n \rightarrow \infty.$
Now, we obtain  
\begin{equation}\label{3.21}
EN_n\left(\frac{1-\eta}{\sigma},\frac{1-\epsilon}{\sigma}\right)=\frac{1}{2\pi\sigma}logn+O(logn)^{\frac{1}{3}.}
\end{equation}
Finally, letting $\frac{1-\epsilon}{\sigma} \leq x \leq 1,$ from (\ref{3.4})-(\ref{3.6}), we have
$$
\frac{(1-3\rho x \sigma)(1-x^{2n}\sigma^{2n})}{(1-\rho x \sigma)(1-x^2\sigma^2)}+ O(\rho^n)< A^2 \leq \sum_{i=0}^{n-1}x^{2i}\sigma^{2i},
$$ 
and 
$$
B^2\leq \sum_{i=0}^{n-1} i^2 x^{2i-2}\sigma^{2i}.
$$
Now,  
$$
\sum_{i=0}^{n-1}\sum_{{j=0},{j\neq i}}^{n-1} \rho^{|i-j|} i x^{i+j-1} \sigma^{i+j}$$
$$=\frac{2n\rho x^{2n-1}\sigma^{2n}}{(\rho-x \sigma)(1-x^2 \sigma^2)}
+\frac{[\rho x^2 \sigma^3(\rho- x\sigma)-\rho x \sigma^2(1-\rho x \sigma)](1-x^{2n}\sigma^{2n})}{(\rho-x \sigma)(1-x^2 \sigma^2)^2(1-\rho x \sigma)}$$
\begin{equation}\label{3.22}
+\frac{[\rho^2 \sigma(\rho-x \sigma)+\rho x^{2n}\sigma^{2n+1}(1-\rho x \sigma)]}{(1-\rho x \sigma)^2(\rho-x \sigma)^2}
+O(\rho^n).
\end{equation}
where 
$\frac{2n\rho x^{2n-1}\sigma^{2n}}{(\rho-x \sigma)(1-x^2 \sigma^2)}< 0.$
Then, 
$$\sum_{i=0}^{n-1}\sum_{{j=0},{j\neq i}}^{n-1} \rho^{|i-j|}i x^{i+j-1} \sigma^{i+j} \leq M_4\frac{(1-x^{2n}\sigma^{2n})}{(1-x^{2}\sigma^{2})^2}.$$
By using the similar method to Farahmand \cite[p.~32]{A18},  we see that for large $n$
\begin{equation} \label{3.22.1}
\frac{\Delta}{A^2} \leq M_5 \sqrt{\frac{n}{1-x \sigma}}. 
\end{equation}
Hence, 
\begin{equation} \label{3.23}
EN_n\left(\frac{1-\epsilon}{\sigma}, 1\right)=O(log n)^{\frac{1}{2}}.
\end{equation}
Therefore, from (\ref{3.14})-(\ref{3.23}), we have 
\begin{equation} \label{3.24}
EN_n(0,1)=\frac{1}{2 \pi \sigma}log n+O(log n)^{\frac{1}{3}}.
\end{equation}
Now, when $x \in [-1,0],$ the same result occur in the interval.
For all sufficiently large $n,$  when $\frac{-1+ \epsilon}{\sigma} \leq x \leq 0$ and $\frac{-1+\eta}{\sigma} \leq x \leq 0,$  $\frac{\Delta}{A^2}$ is similar to (\ref{3.11}) and we have 
\begin{equation} \label{3.25}
EN_n\left(\frac{-1+\eta}{\sigma},0\right)=O(log n)^{\frac{1}{3}}.
\end{equation}
Again, when $\frac{-1+\epsilon}{\sigma} \leq x \leq \frac{-1+\eta}{\sigma},$ simialr to (\ref{3.21}), we get 
\begin{equation} \label{3.26}
EN_n\left(\frac{-1+\epsilon}{\sigma},\frac{-1+\eta}{\sigma}\right)=\frac{1}{2 \pi \sigma}log n+ O(log n)^{\frac{1}{3}}.
\end{equation}

Now, for $-1 \leq x \leq \frac{-1+\epsilon}{\sigma},$
the $var(\sum_{i=0}^{n-1}a_i |x|^i \sigma^{2i}) \geq 0.$
Hence $$|\sum_{i=0}^{n-1}\sum_{j=0}^{n-1} \rho
^{|i-j|}{ x^{i+j} \sigma^{i+j}}| \leq \sum_{i=0}^{n-1} x^{2i} \sigma^{2i},$$
and from (\ref{3.1}), we get 
\begin{equation} \label{3.27}
A^2 \leq 2 \sum _{i=0}^{n-1} x^{2i} \sigma^{2i}.
\end{equation}
Similarly by (\ref{3.27}), we have $$B^2 \leq 2 \sum _{i=0}^{n-1} i^2 x^{2i-2} \sigma^{2i}.$$ 
This gives that 
$$\Delta^2 < 4[(\sum _{i=0}^{n-1} x^{2i} \sigma^{2i})(\sum _{i=0}^{n-1} i^2 x^{2i-2} \sigma^{2i})-(\sum _{i=0}^{n-1} i x^{2i-1} \sigma^{2i})^2]$$
\begin{equation} \label{3.28}
+ 3(\sum _{i=0}^{n-1} i x^{2i-1} \sigma^{2i})^2+ 2(\sum _{i=0}^{n-1} i x^{2i-1} \sigma^{2i})(\sum_{i=0}^{n-1}\sum_{j=0}^{n-1} 
\rho^{|i-j|}{i x^{i+j-1} \sigma^{i+j}}).
\end{equation}

Also, for $x$ belonging to this interval, we have 
$$\sum_{i=0}^{n-1}\sum_{j=0}^{n-1} \rho
^{|i-j|} x^{i+j} \sigma^{i+j}= \frac{2 \rho^2 x(\sigma- x^{2n}\sigma^{2n})-2 \rho x^2 (\sigma^2- x^{2n-2}\sigma^{2n-2})}{(1-x^2 \sigma^2)(\rho- x \sigma)(1- \rho x  \sigma)}+O(\rho^n)<0.
$$
From (\ref{3.1}), we obtain
$$A^2 > \sum _{i=0}^{n-1} x^{2i} \sigma^{2i}.$$
Using similar way of (\ref{3.22.1}), from (\ref{3.28}), we can show that 
$$\frac{\Delta}{A^2} \leq M_6 \sqrt{\frac{n}{1+x \sigma}},$$
and therefore 
\begin{equation} \label{3.29}
EN_n\left(-1,\frac{-1+\epsilon}{\sigma}\right)=O(log n)^{\frac{1}{2}}.
\end{equation}
Hence, from (\ref{3.25}), (\ref{3.26}) and (\ref{3.29}), we get 
\begin{equation} \label{3.30}
EN_n(-1,0)=\frac{1}{2 \pi \sigma} log n + O(log n)^{\frac{1}{2}}.
\end{equation}
Hence finally, from (\ref{3.24}) and (\ref{3.30}), we get
$$
EN_n(-\infty, \infty)=\frac{2}{\pi \sigma}log n+ O(log n)^{\frac{1}{2}},
$$ 
which completes proof of the theorem.\\
\end{proof}

\textbf{Acknowledgments}\\
This research work was supported by University Grant Comission (National Fellowship with letter no-F./2015-16/NFO-2015-17-OBC-ORI-33062).


\begin{thebibliography}{ABCR}
\bibitem{A14} A. T. Bharucha-Reid, M. Sambandham, Random polynomials, Academic Press, New York,         1986.  

\bibitem{A4} J. E. A. Dunnage, The number of real zeros of a random trigonometric polyomial, \textit{Proc. Lond. Math. Soc.,} \textbf{16} (1996) 53-84.

\bibitem{A18} K. Farahmand, Topic in random polynomials, Addison-Wesley Longman, London, 1998. 

\bibitem{A16} K. Farahmand, A. Nezakati, Algebric polynomials with dependent random coefficients, \textit{Commun. Appl. Anal.,} \textbf{9} (2005) 95-104.

\bibitem{A3} K. Farahmand, A. Nezakati, Algebric polynomials with nonidentical random coefficients, \textit{Proc. Amer. Math. Soc.,} \textbf{133} (2004) 275-283-.

\bibitem{A20} K. Farahmand, A. Nezakati, Real zeros of algebric polynomials with dependent random coefficients, \textit{Stoch. Anal. Appl.,} \textbf{28} (2010) 558-564.

\bibitem{A2} I. A. Ibragimov, N. B. Maslova, On the exepted number of real zeros of random polynomials II, coefficients with nonzero means, \textit{Theory Probab. Appl.,} \textbf{16} (1971) 485-493. 

\bibitem{A1} M. Kac, On the average number of real roots of a random algebric equation, \textit{Bull. Amer. Math. Soc.,} \textbf{49} (1943) 314-320. 

\bibitem{A15} C. Newman, Asymptotic Independence and Limit Theorems for Positively and Negatively Dependent Random Variables, \textit{Inequalities in Statistics and Probability, IMS Lecture Notes- Monograhp Series,} \textbf{5} (1984) 127-140.

\bibitem{A9} N. Renganathan, M. Sambandham, On the average number of real zeros of a random trigonometric poiynomial with dependent coefficients-II,      \textit{Indian J. Pure Appl. Math.,} \textbf{15} (1984) 951-956.

\bibitem{A17} S.O. Rice, Mathematical theory of random noise, \textit{Bell System Tech. J.,}  \textbf{25} (1945), 46-156. Reprinted in :Selected Papers on Noise and Stochastic Processes(ed. N. Wax), Dover, New York, 1954, 133-294.

\bibitem{A5} M. Sambandham, On a random algebric polynomial, \textit{J. Indian Math. Soc.,} \textbf{41} (1977) 83-97.

\bibitem{A7} M. Sambandham, On the average number of real zeros of a class of random algebric curves, \textit{Pacific J. Math.,} \textbf{81} (1979) 207-215.

\bibitem{A8} M. Sambandham, On the number of real zeros of a random trigonometric polynomial, \textit{Trans. Amer. Math. Soc.,} \textbf{238} (1978) 57-70.

\bibitem{A10} M. Sambandham, On the real roots of the random algebric polynomial, \textit{Indian J. Pure Appl. Math.,} \textbf{7} (1976) 1062-1070.

\bibitem{A6} M. Sambandham, On the upper bound of the number of real roots of a random algebric equation, \textit{J. Indian Math. Soc.,} \textbf{42} (1978) 15-26.


\end{thebibliography}
\end{document}